\documentclass[11pt]{amsart}
\usepackage{graphicx,amscd,color}
\usepackage{a4wide}

\theoremstyle{plain}
\newtheorem{theorem}{Theorem}[section]

\newtheorem{lemma}[theorem]{Lemma}

\theoremstyle{remark}

\begin{document}

\title [On weak reducing disks for the unknot in $3$-bridge position]
{On weak reducing disks for the unknot\\ in $3$-bridge position}

\author[B. Kwon]{Bo-hyun Kwon}
\address{Department of Mathematics, Korea University, Seoul, Korea}
\email{bortire74@gmail.com}

\author[J. H. Lee]{Jung Hoon Lee}
\address{Department of Mathematics and Institute of Pure and Applied Mathematics,
Chonbuk National University, Jeonju 54896, Korea}
\email{junghoon@jbnu.ac.kr}


\begin{abstract}
We show that the complex of weak reducing disks
for the unknot in $3$-bridge position is contractible.
\end{abstract}

\maketitle

\section{Introduction}\label{sec1}

Let $S^3$ be decomposed into two $3$-balls $V$ and $W$ with common boundary sphere $S$.
Let $K$ be an unknot in $n$-bridge position with respect to $S$.
That is,
$V \cap K$ and $W \cap K$ are collections of $n$ boundary parallel arcs in $V$ and $W$ respectively.
We call $(S^3, K) = (V, V \cap K) \cup_S (W, W \cap K)$ a (genus-$0$) $n$-bridge splitting and
$S$ an $n$-bridge sphere.
Each arc of $V \cap K$ and $W \cap K$ is called a bridge.

An $n$-bridge sphere is unique for every $n$ \cite{Otal},
but to understand the mapping class group or deep structure on topological minimality
we need to study the disk complex of the bridge sphere.

The {\em disk complex} $\mathcal{D}$ for $V - K$ is a simplicial complex defined as follows.

\begin{itemize}
\item Vertices of $\mathcal{D}$ are isotopy classes of compressing disks for $S - K$ in $V - K$.
\item A collection of $k+1$ vertices forms a $k$-simplex
if there are representatives for each that are pairwise disjoint.
\end{itemize}

Some important complexes equivalent to subcomplexes of a disk complex
such as sphere complex, primitive disk complex are proved to be useful
to understand the genus two Heegaard splitting of $S^3$ \cite{Scharlemann}, \cite{Cho}.
In this paper, we consider the complex of weak reducing disks for the unknot in $3$-bridge position.
The complex of weak reducing disks is interesting in that
it is equivalent to the complex of cancelling disks (Lemma \ref{lem1}),
which is reminiscent of the primitive disk complex mentioned above.
It is a natural question whether the complex is connected and contractible or not.
Using a criterion in \cite{Cho}, we show that it is contractible.

\begin{theorem}\label{thm1}
The complex of weak reducing disks for the unknot in $3$-bridge position is contractible.
\end{theorem}

In Section \ref{sec2}, we consider the relationship between weak reducing disks and cancelling disks
and some necessary lemmas.
In Section \ref{sec3}, we give a proof of Theorem \ref{thm1}.

\section{Weak reducing disks and cancelling disks}\label{sec2}

Let $K$ be an unknot in $3$-bridge position with respect to $V \cup_S W$.
A properly embedded disk $D$ in $V - K$ is a {\em compressing disk} if
$\partial D$ is an essential simple closed curve in $S - K$.
An embedded disk $\Delta$ in $V$ is a {\em bridge disk} if
$\partial \Delta$ is a union of two arcs $a$ and $b$ with $a \cap b = \partial a = \partial b$, where
$a = \Delta \cap K$ and $b = \Delta \cap S$.
A compressing disk in $W - K$ and a bridge disk in $W$ are defined similarly.

A compressing disk $D$ cuts off a $3$-ball $B$, which contains one bridge, from $V$.
We can take a unique bridge disk $\Delta$ in $B$.
Conversely, for a bridge disk $\Delta$,
the frontier of a neighborhood of $\Delta$ in $V$ is a compressing disk $D$.
So there is a one-to-one correspondence
between the set of compressing disks $\mathcal{C}$ and the set of bridge disks $\mathcal{B}$.
Let $d: \mathcal{C} \to \mathcal{B}$ be the bijection defined by $d(D) = \Delta$.
A similar bijection $\overline{d} : \overline{\mathcal{C}} \to \overline{\mathcal{B}}$ exists between
the set of compressing disks $\overline{\mathcal{C}}$ in $W - K$ and
the set of bridge disks $\overline{\mathcal{B}}$ in $W$.

A compressing disk $D$ is a {\em weak reducing disk} if
there is a compressing disk $E \subset W - K$ such that $\partial D \cap \partial E = \emptyset$.
A bridge disk $\Delta$ is a {\em cancelling disk} if
there is a bridge disk $\overline{\Delta} \subset W$ such that
$\Delta \cap \overline{\Delta} = \{$one point of $K \}$.
Here $(\Delta, \overline{\Delta})$ is called a {\em cancelling pair}.
Figure \ref{fig1} illustrates weak reducing disks and cancelling disks.

\begin{figure}[ht!]
\begin{center}
\includegraphics[width=11cm]{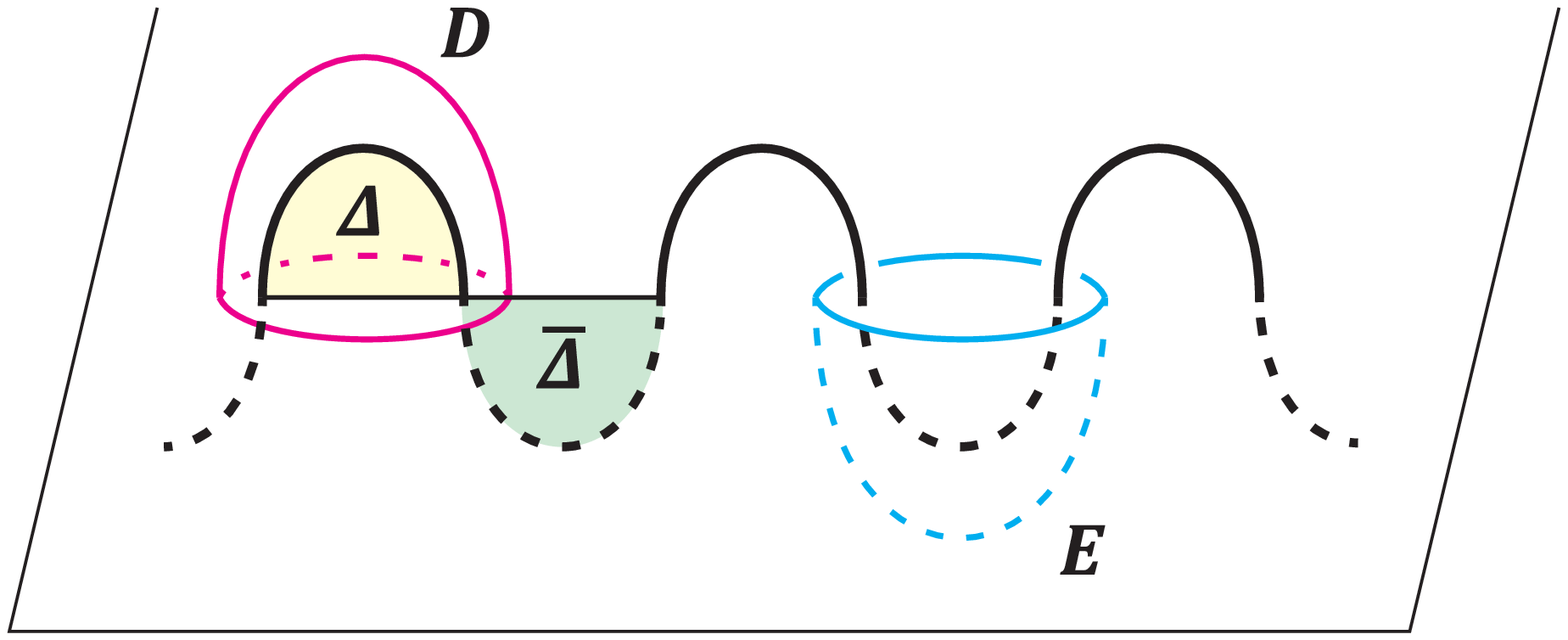}
\caption{Weak reducing disks and cancelling disks}\label{fig1}
\end{center}
\end{figure}

Let $\Delta$ be a bridge disk in, say $V$, with $\Delta \cap K = a$ and $\Delta \cap S = b$.
An isotopy of $a$ to $b$ along $\Delta$ and further, slightly into $W$ is called a {\em reduction}.
See Figure \ref{fig2}.
A reduction along $\Delta$ yields a $2$-bridge position of $K$ if and only if
there is a bridge disk $\overline{\Delta}$ in $W$ such that
$(\Delta, \overline{\Delta})$ is a cancelling pair \cite{Lee}.

\begin{figure}[ht!]
\begin{center}
\includegraphics[width=14.2cm]{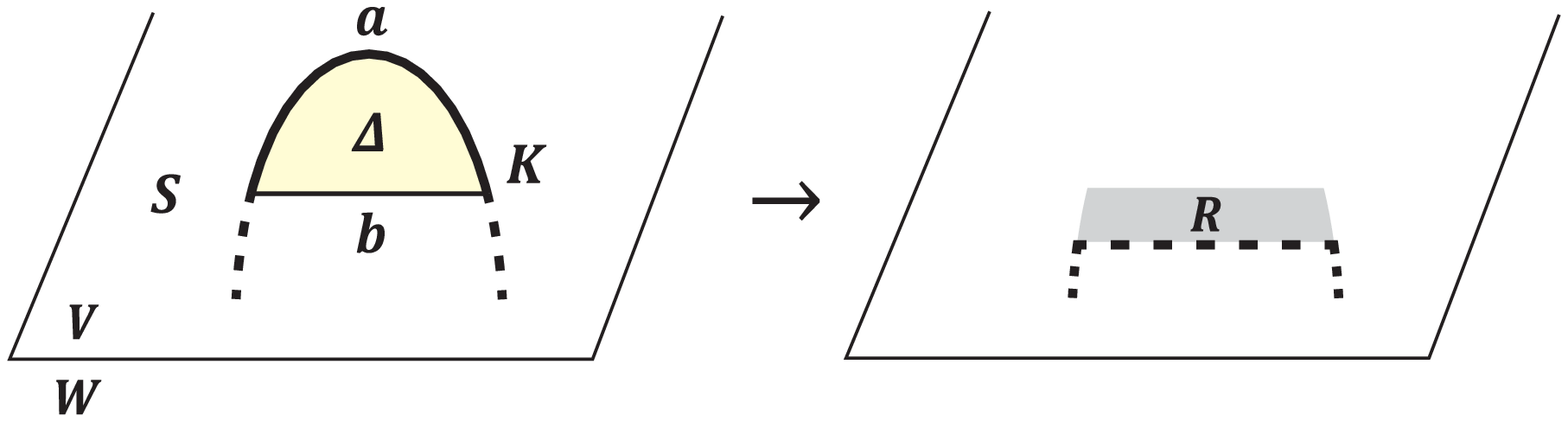}
\caption{A reduction along $\Delta$}\label{fig2}
\end{center}
\end{figure}

\begin{lemma}\label{lem1}
A compressing disk $D$ is a weak reducing disk if and only if $\Delta$ is a cancelling disk.
\end{lemma}

\begin{proof}
Suppose that $E \subset W - K$ is a compressing disk disjoint from $D$.
Let $\Delta = d(D)$ and $\Delta' = \overline{d}(E)$.
A simultaneous reduction along $\Delta$ and $\Delta'$ results in a $1$-string decomposition of $K$,
i.e. $K = K_1 \# K_2$ for some knots $K_1$ and $K_2$.
Because $K$ is an unknot, both $K_1$ and $K_2$ are unknots \cite{Rolfsen}.
Thus we can easily see that a reduction along the single $\Delta$ results a $2$-bridge position of $K$.
By \cite[Theorem $1.1$]{Lee}, there exists a bridge disks $\overline{\Delta}$ in $W$ such that
$(\Delta, \overline{\Delta})$ is a cancelling pair.

Conversely, suppose that $(\Delta, \overline{\Delta})$ is a cancelling pair.
A cancellation along $\Delta \cup \overline{\Delta}$ yields a $2$-bridge position of $K$.
Take a compressing disk $E$ in $W - K$ for the $2$-bridge position.
We recover the original $3$-bridge position by giving a perturbation using $\Delta$ and $\overline{\Delta}$:
it is done in such a way that $E$ is disjoint from $\Delta \cup \overline{\Delta}$.
Hence $D = d^{-1}(\Delta)$ is a weak reducing disk.
\end{proof}

\begin{lemma}\label{lem2}
For a weak reducing disk $D \subset V - K$,
there exists a weak reducing disk $D' \subset V - K$ that is disjoint from $D$.
\end{lemma}

\begin{proof}
Let $\Delta = d(D)$ be the cancelling disk corresponding to $D$ and
let $(\Delta, \overline{\Delta})$ be a cancelling pair.
Since $(\Delta, \overline{\Delta})$ is a cancelling pair,
we can take a compressing disk $D' \subset V - K$ disjoint from $\Delta \cup \overline{\Delta}$
as in the proof of Lemma \ref{lem1}.
Then $D'$ is a weak reducing disk because it is disjoint
from the compressing disk $\overline{d}^{-1}(\overline{\Delta})$ in $W - K$, and
$D'$ is disjoint from $D$ also.
\end{proof}

The following fact is well known, so we omit the proof here.

\begin{lemma}\label{lem3}
There is a unique compressing disk (up to isotopy) for a knot in $2$-bridge position.
\end{lemma}

A collection of six cancelling disks $\{\Delta_1, \Delta_2, \ldots, \Delta_6\}$
is called a {\em complete cancelling disk system} if
$(\Delta_i, \Delta_{i+1})$ ($i=1,\ldots,5$) and $(\Delta_6, \Delta_1)$ are cancelling pairs and
$\Delta_i \cap \Delta_j = \emptyset$ for other pairs of $\{ i,j \}$.

We remark that a complete cancelling disk system consisting of four cancelling disks
for an unknot in $2$-bridge position can be defined similarly.

\begin{lemma}\label{lem4}
If $\Delta$ and $\Delta'$ are disjoint cancelling disks in $V$,
then $\{ \Delta, \Delta' \}$ extends to a complete cancelling disk system (Figure \ref{fig3}).
\end{lemma}

\begin{figure}[ht!]
\begin{center}
\includegraphics[width=10cm]{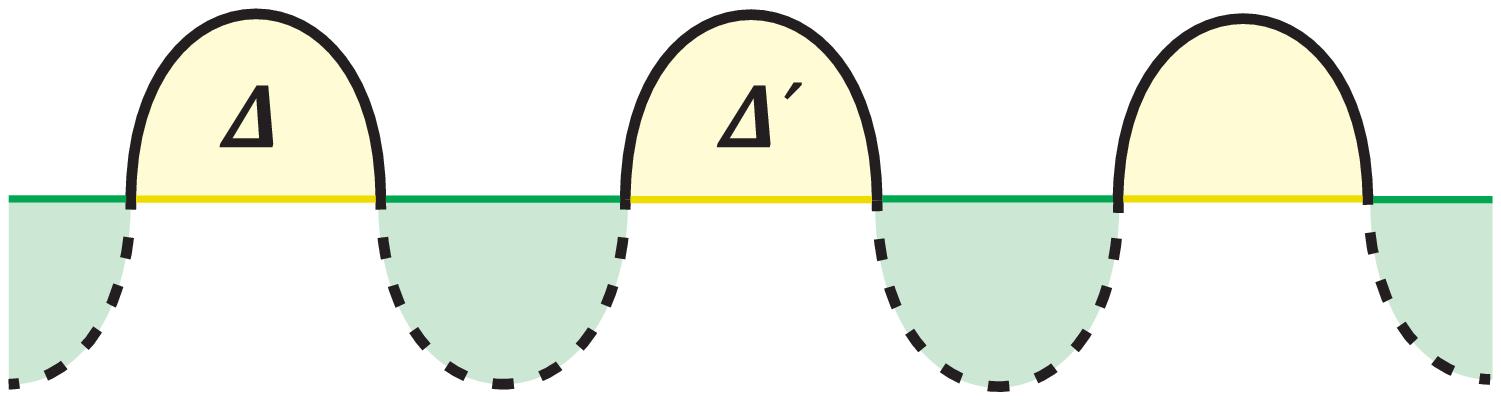}
\caption{A complete cancelling disk system extending $\{ \Delta, \Delta' \}$}\label{fig3}
\end{center}
\end{figure}

\begin{proof}
Since $\Delta$ is a cancelling disk,
a reduction along $\Delta$ yields an unknot in $2$-bridge position.
After the reduction, there exists a cancelling pair $( \Delta', \overline{\Delta'} )$
by uniqueness of the compressing disk and bridge disk
for a knot in $2$-bridge position (Lemma \ref{lem3}).
By \cite{Lee}, for any bridge in $W$ a bridge disk can be chosen so that
it does not intersect the rectangular region $R$ in $W$
(as in Figure \ref{fig2} and Figure \ref{fig4})
that are created after the reduction along $\Delta$.
In particular, $\overline{\Delta'}$ does not intersect $R$.
Then do the reverse operation of the reduction along $\Delta$.
We have disjoint $\Delta$ and $( \Delta', \overline{\Delta'})$.

By cancelling along $( \Delta', \overline{\Delta'} )$ as in Figure \ref{fig4},
we get a $2$-bridge unknot.
Take a complete cancelling disk system containing $\Delta$ for the $2$-bridge position.
Then do the reverse isotopy of the cancellation,
i.e. a perturbation,
to get the cancelling pair $( \Delta', \overline{\Delta'}) $ back additionally.
We obtained a complete cancelling disk system
containing $\{ \Delta, \Delta', \overline{\Delta'} \}$ for the original $3$-bridge position.
\end{proof}

\begin{figure}[ht!]
\begin{center}
\includegraphics[width=9.8cm]{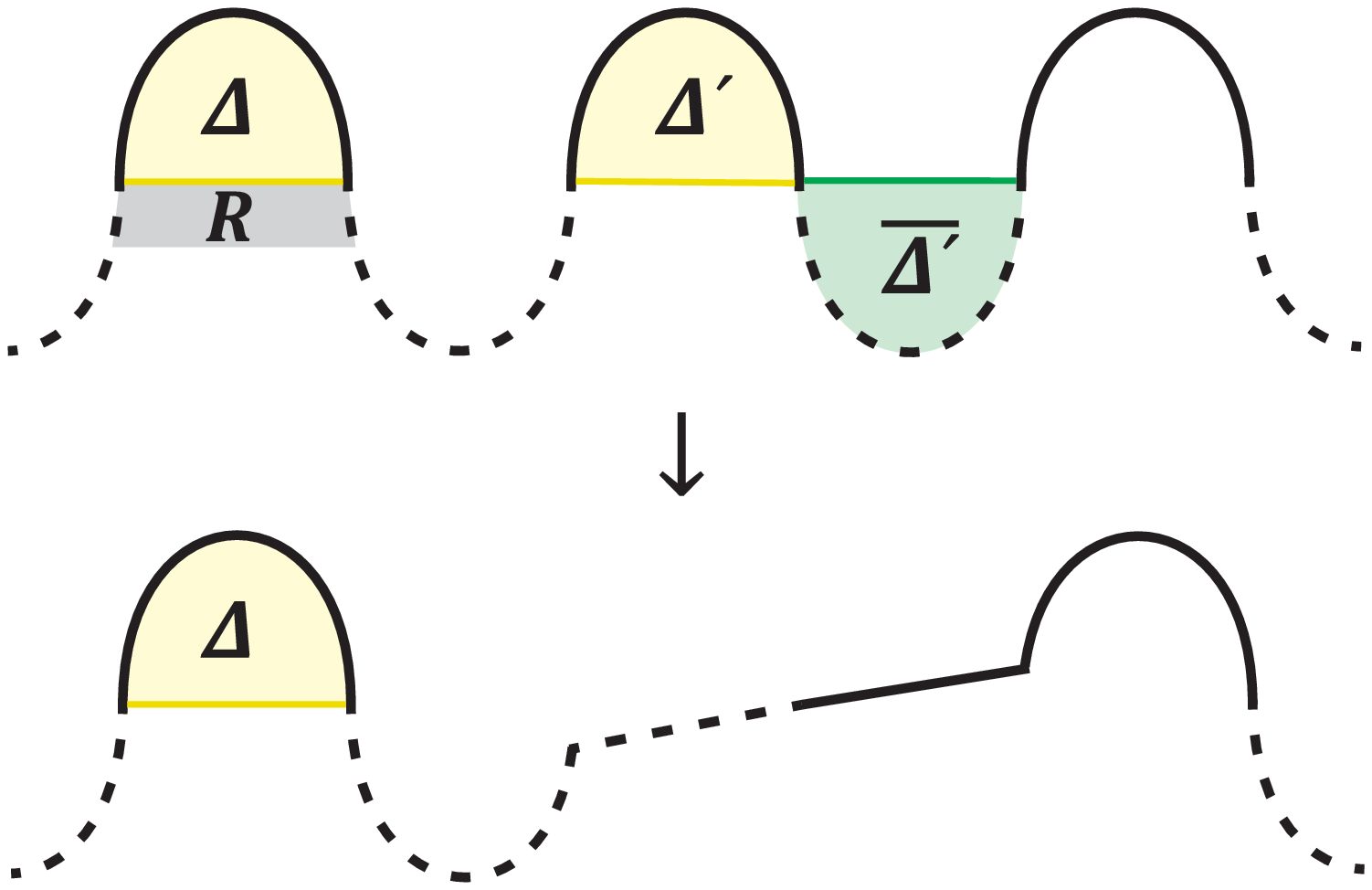}
\caption{A cancellation along $( \Delta', \overline{\Delta'} )$}\label{fig4}
\end{center}
\end{figure}

Let $\Delta_x, \Delta_y, \Delta_z$ be three disjoint bridge disks in $W$.
Let $x = \Delta_x \cap S$, $y = \Delta_y \cap S$, and $z = \Delta_z \cap S$.
By removing a small open neighborhood $\eta (K)$ of $K$ from $W$, we obtain a genus three handlebody.
In order to prove the main theorem in the coming section,
we would like to express a simple closed curve $\gamma$ in $S - K$
as a word in terms of the three generators of $\pi_1 (W - \eta (K))$.
An oriented loop in $S - K$ passing one of the simple arcs $x, y, z$ once
becomes a representative of the corresponding generator of
$\pi_1 (W - \eta (K)) \simeq \pi_1 (W - K)$.
Therefore, the simple closed curve $\gamma$ can be represented by a word $w$ in $x, y, z$.

\begin{lemma}\label{lem5}
If an essential simple closed curve $\gamma$ in $S - K$ bounds a disk in $W - K$,
then the word $w$ of $\gamma$ is reduced to an empty word.
\end{lemma}

\begin{proof}
Suppose that $w$ is not reduced to an empty word.
Then $w$ represents a non-trivial element in the free group $\pi_1(W - K)$.
It contradicts that $\gamma$ bounds a disk in $W - K$.
\end{proof}

\begin{lemma}\label{lem6}
If an essential simple closed curve $\delta$ contains
three subarcs connecting $x$ to $y$, $y$ to $z$, and $z$ to $x$,
then there is no compressing disk in $W - K$ that is disjoint from $\delta$.
\end{lemma}

\begin{figure}[ht!]
\begin{center}
\includegraphics[width=6cm]{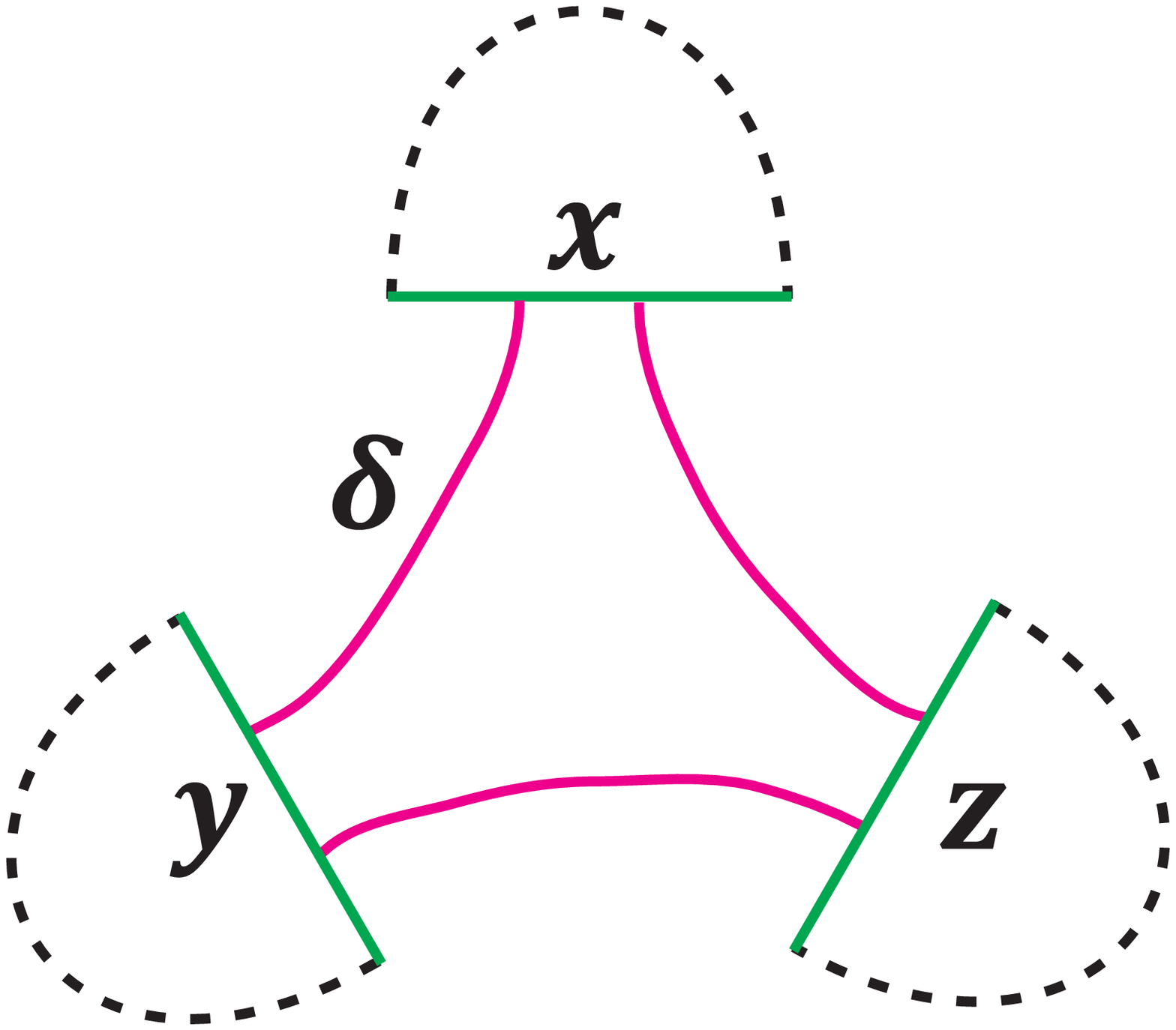}
\caption{An obstruction to a weak reducing disk in $W - K$}\label{fig5}
\end{center}
\end{figure}

\begin{proof}
Let $E$ be a compressing disk in $W - K$.
If $E$ is disjoint from $\Delta_x \cup \Delta_y \cup \Delta_z$,
then $\partial E$ is a simple closed curve that does not bound a disk in $P = S - (x \cup y \cup z)$,
which is homeomorphic to a $3$-punctured sphere.
Then the three subarcs of $\delta$ is an obstruction for $\partial E$ to be disjoint from $\delta$.
See Figure \ref{fig5}.
Suppose that $E$ intersects $\Delta_x \cup \Delta_y \cup \Delta_z$.
We may assume that $E \cap (\Delta_x \cup \Delta_y \cup \Delta_z)$ consists of arc components.
Consider an outermost disk $C$ in $E$
cut off by an outermost arc of $E \cap (\Delta_x \cup \Delta_y \cup \Delta_z)$.
Then $C \cap P$ is an essential arc in $P$ such that
the two endpoints of $C \cap P$ are on the same component of $\partial P$.
Again, the three subarcs of $\delta$ is an obstruction to $\partial E$.
\end{proof}

\section{Proof of Theorem \ref{thm1}}\label{sec3}

\subsection{Setting}

Let $D$ and $F$ be weak reducing disks in $V - K$ such that $D \cap F \ne \emptyset$.
We assume that $|D \cap F|$ is minimal up to isotopy.
By \cite[Theorem $4.2$]{Cho}, it is enough to show that
one of the disks obtained by surgery of $D$ along an outermost disk in $F$
cut off by an outermost arc of $D \cap F$ is a weak reducing disk.
Let $\Delta = d(D)$.
By an isotopy of $D$, we assume that
$\partial D$ equals the boundary of a small neighborhood of $\Delta \cap S$ in $S$.
Let $\alpha_o$ be an outermost arc of $D \cap F$ in $F$ and
let $C$ be the corresponding outermost disk in $F$ cut off by $\alpha_o$.
Let $\alpha = C \cap S$.

First we consider the special case that
there exists another weak reducing disk $D'$ in $V - K$ disjoint from $D$ and $C$.
Let $\Delta' = d(D')$.
We isotope $D'$ so that $\partial D'$ equals
the boundary of a small neighborhood of $\Delta' \cap S$ in $S$.
Let $P$ and $P'$ be the disks that $\partial D$ and $\partial D'$ bound,
containing $\Delta \cap S$ and $\Delta' \cap S$ respectively.
The arc $\alpha$ cuts off an annulus $A$ from $S -(P \cup P')$.
Let $\beta$ be an essential arc of $A$ disjoint from $\alpha$.
We give $\beta$ an orientation from $\partial D'$ to $\partial D$.
We isotope $\alpha$ so that
$\alpha$ equals the frontier of a small neighborhood of $\beta \cup P'$ in ${\mathrm{cl}}(S - P)$.
See Figure \ref{fig6}.

\begin{figure}[ht!]
\begin{center}
\includegraphics[width=5.5cm]{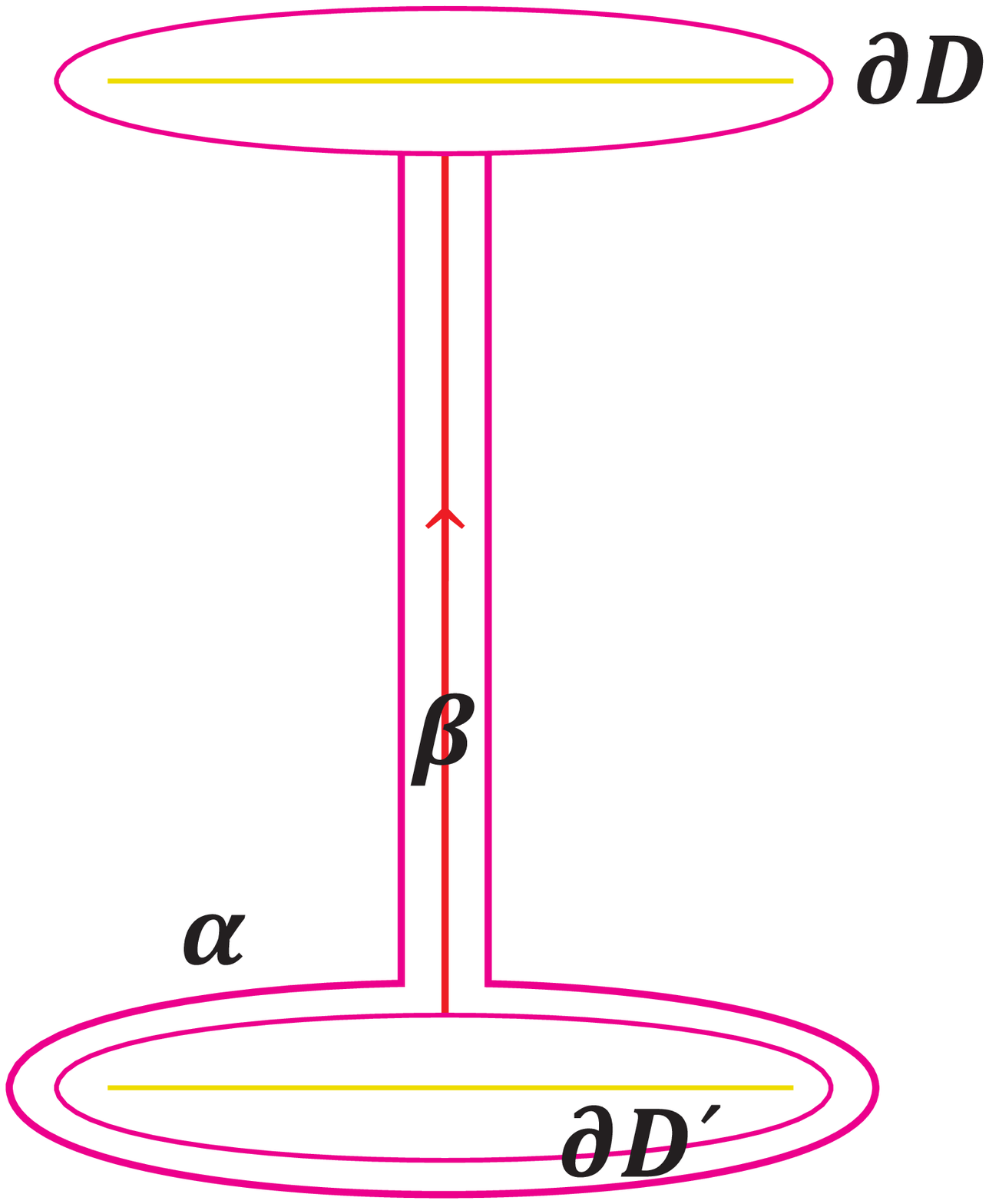}
\caption{The arcs $\alpha$ and $\beta$}\label{fig6}
\end{center}
\end{figure}

Using Lemma \ref{lem4}, extend $\{ \Delta, \Delta' \}$ to a complete cancelling disk system
$\{ \Delta, \Delta' ,\Delta_b, \Delta_x, \Delta_y, \Delta_z \}$, where
$\Delta, \Delta', \Delta_b$ are in $V$ and $\Delta_x, \Delta_y, \Delta_z$ are in $W$.
Let $b = \Delta_b \cap S$, $x = \Delta_x \cap S$, $y= \Delta_y \cap S$, $z = \Delta_z \cap S$.
We assume that the arcs $x,y,z,b$ satisfy the following (Figure \ref{fig7}).

\begin{itemize}
\item $x$ intersects only a point of $\partial D$.
\item $y$ intersects only a point of $\partial D'$.
\item $z$ intersects a point of $\partial D$ and a point of $\partial D'$.
\item $b$ intersects none of $\partial D$ and $\partial D'$.
\end{itemize}

\begin{figure}[ht!]
\begin{center}
\includegraphics[width=6cm]{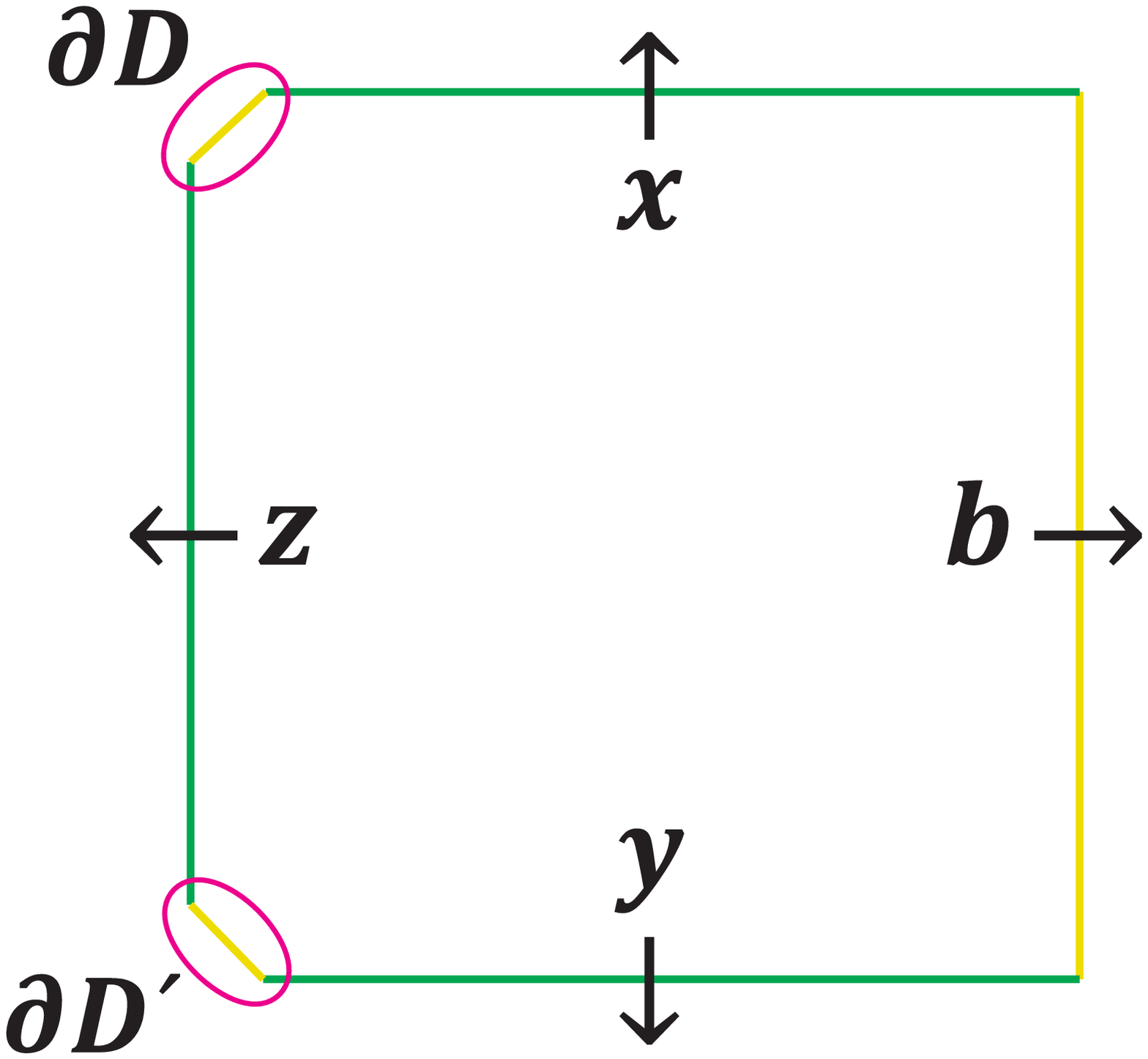}
\caption{The arcs $x, y, z$, and $b$}\label{fig7}
\end{center}
\end{figure}

We shrink each of $P$ and $P'$ to a point.
Then $\beta$ can be regarded as a properly embedded arc in a $4$-punctured sphere.
An isotopy class of a  properly embedded arc with different endpoints in a $4$-punctured sphere
is completely determined by its slope.
A slope $s \in \mathbb{Q} \cup \{ \infty \}$ is expressed
as a fraction $\frac{p}{q}$ of relatively prime integers $p$ and $q$.
The oriented arc $\beta$ starts at the lower left vertex and ends at the upper left vertex.
Note that for the slope $s = \frac{p}{q}$ of $\beta$, $p$ is odd and $q$ is even.
See Figure \ref{fig8} for an example of $s = \frac{3}{8}$.
Without loss of generality, we may assume that $s$ is positive.

\begin{figure}[htb]
\begin{center}
\includegraphics[width=8.3cm]{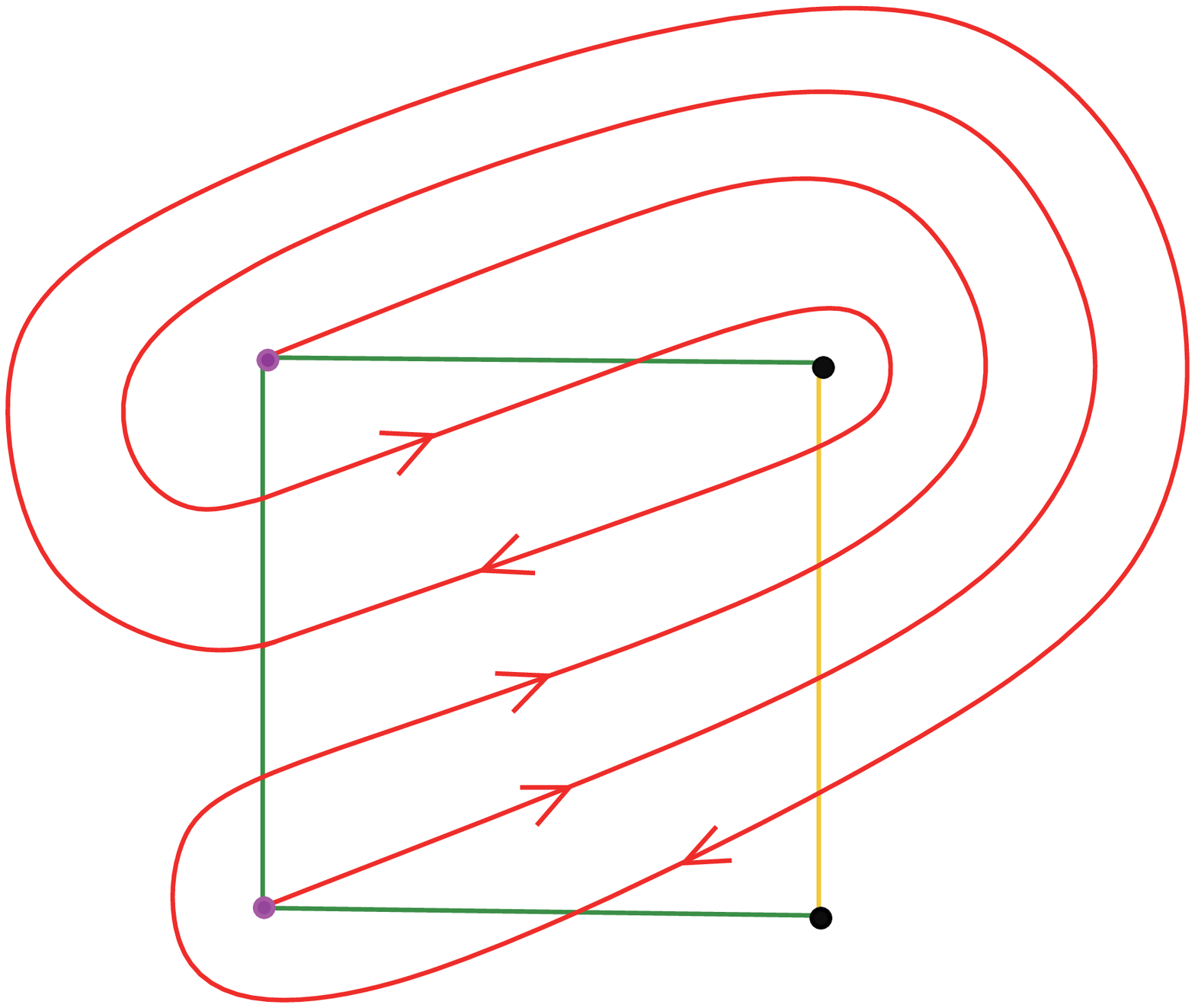}
\caption{An arc of slope $\frac{3}{8}$}\label{fig8}
\end{center}
\end{figure}

Each time $\beta$ passes through an arc, it is given a generator
among $x^{\pm 1}, y^{\pm 1}, z^{\pm 1}, b^{\pm 1}$.
So the arc $\beta$ can be written as a reduced word $w$ in $x,y,z,b$.
We divide cases according to the slope $s = \frac{p}{q}$ and investigate the word $w$.
We show that either a surgery of $D$ along the outermost disk $C$ yields a weak reducing disk,
or actually $F$ is not a weak reducing disk.

When $s = \infty (= \frac{1}{0})$, $\beta$ is equal to $z$.
In this case, the disk obtained by a surgery of $D$ along $C$ is a weak reducing disk.
If $s > 1$, then $\alpha$ contains all three types of subarcs
connecting $x$ to $y$, $y$ to $z$, and $z$ to $x$.
So a compressing disk in $W - K$ disjoint from $F$ cannot exist by Lemma \ref{lem6}.
See Figure \ref{fig9}.
The slope $s$ cannot be $1(= \frac{1}{1})$.
Hence we assume that $0 < s < 1$.

\begin{figure}[ht!]
\begin{center}
\includegraphics[width=6cm]{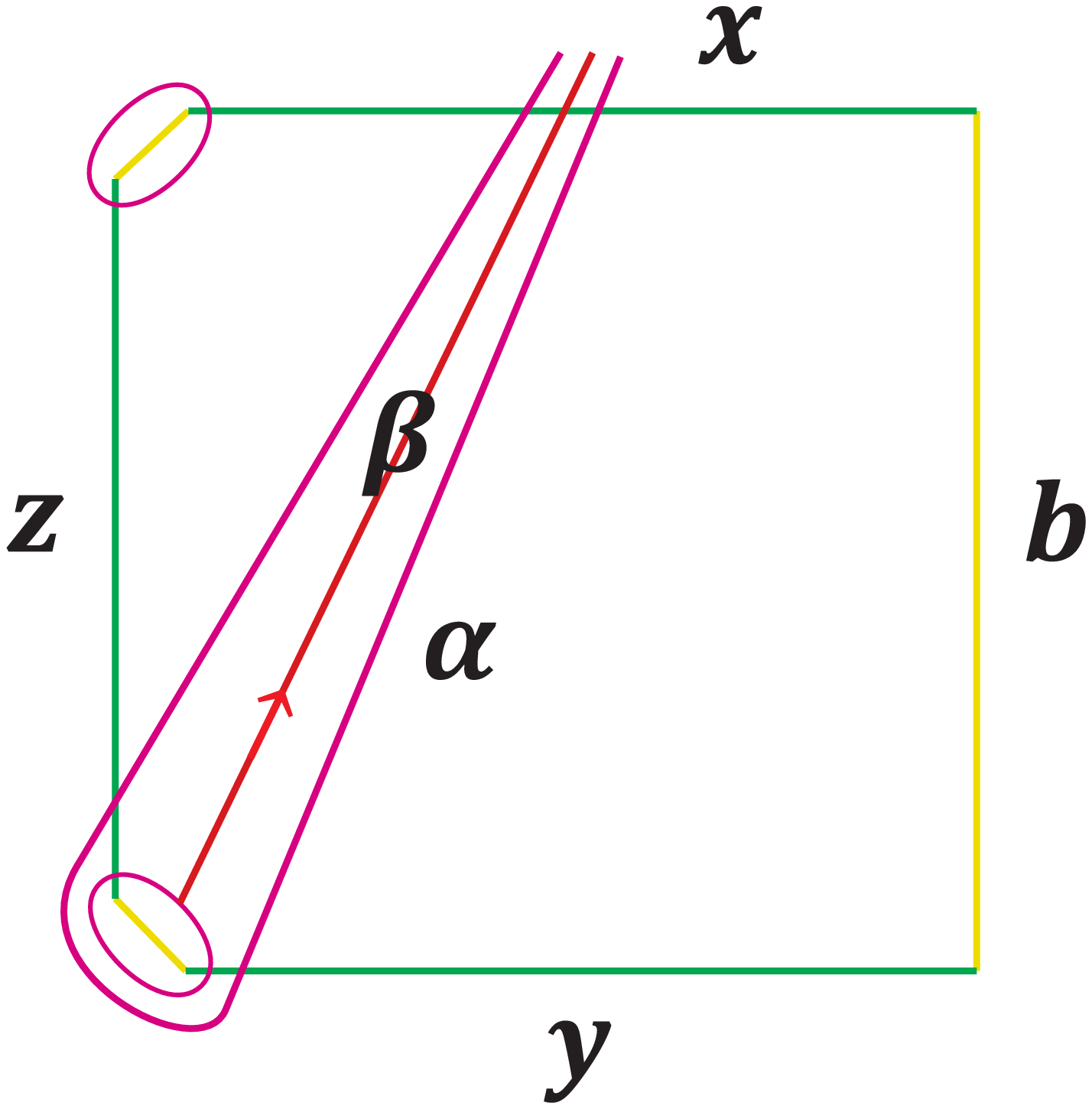}
\caption{The slope $s > 1$}\label{fig9}
\end{center}
\end{figure}

\begin{figure}[htb]
\begin{center}
\includegraphics[width=8.5cm]{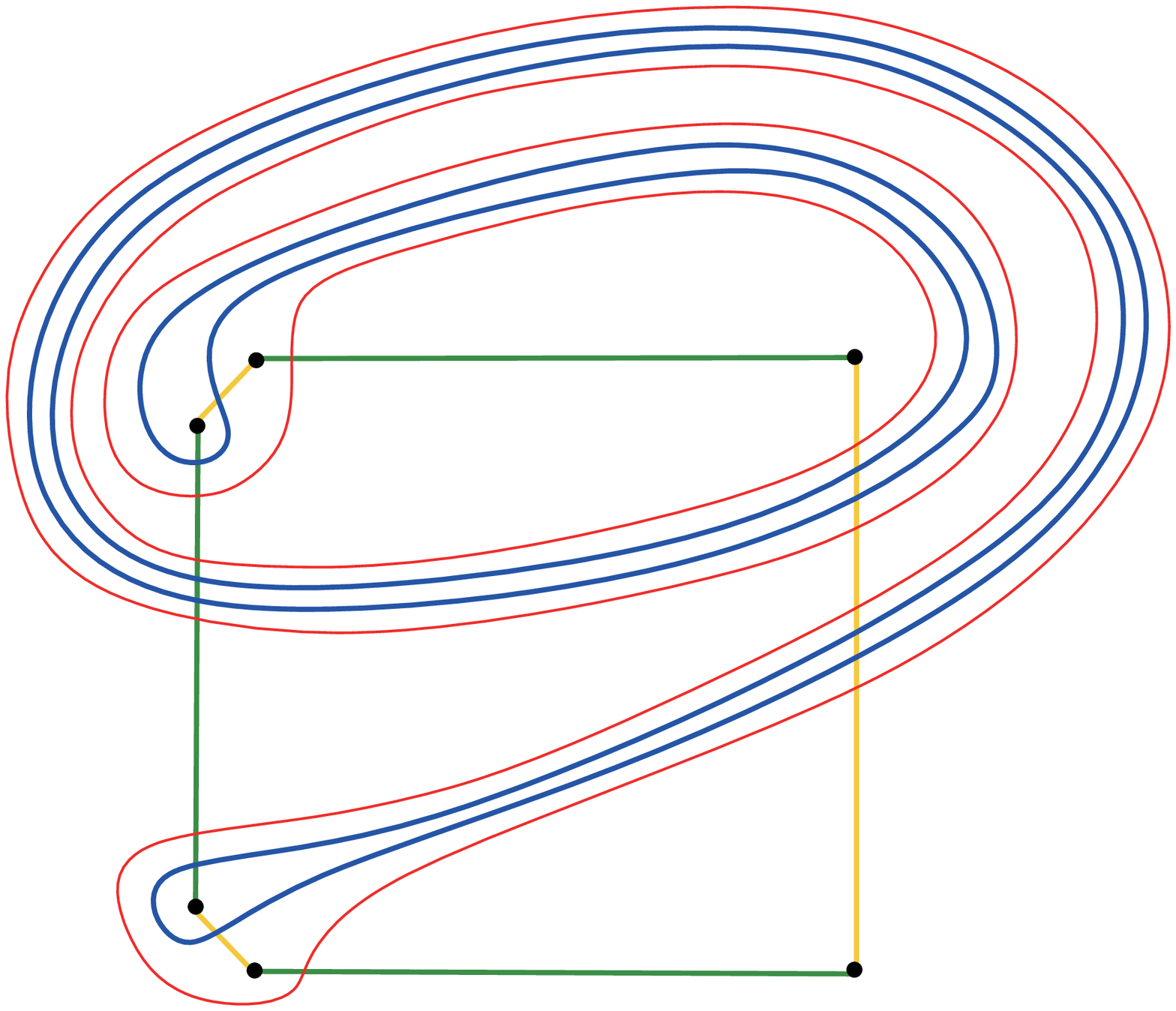}
\caption{A weak reducing pair after the surgery}\label{fig10}
\end{center}
\end{figure}

If $\frac{1}{2} < s < 1$,
then $w$ begins with $bx^{-1}$, so $\alpha$ contains all three types of subarcs as above.
Suppose that $s = \frac{1}{2k}$ for some natural number $k$.
Then we can see that a surgery of $D$ along $C$ yields a weak reducing disk.
See Figure \ref{fig10} for an example of $k = 2$.
Thus from now on, for $s = \frac{p}{q}$ with $p$ odd and $q$ even,
we assume that $\frac{2p}{q} < 1$ and $p > 1$.
For such a slope $s$, there exists a natural number $k$ satisfying one of the following.

Case (a). $\frac{2kp}{q} < 1 < \frac{(2k+1)p}{q}$

Case (b). $\frac{(2k+1)p}{q} < 1 < \frac{(2k+2)p}{q}$

For each case, there are two subcases according to when a multiple of $\frac{p}{q}$ exceeds $2$.
We list the beginning subword of $w$ in each subcase together.

Subcase $1$. $\frac{4kp}{q} < 2 < \frac{(4k+1)p}{q}$,
$w$ begins with $w_1 = (bz^{-1})^k x (b^{-1}z)^k y^{-1}b$.

Subcase $2$. $\frac{(4k+1)p}{q} < 2 < \frac{(4k+2)p}{q}$,
$w$ begins with $w_2 = (bz^{-1})^k x (b^{-1}z)^k b^{-1}y$.

Subcase $3$. $\frac{(4k+2)p}{q} < 2 < \frac{(4k+3)p}{q}$,
$w$ begins with $w_3 = (bz^{-1})^k b x^{-1} z (b^{-1}z)^k y^{-1}b$.

Subcase $4$. $\frac{(4k+3)p}{q} < 2 < \frac{(4k+4)p}{q}$,
$w$ begins with $w_4 = (bz^{-1})^k b x^{-1} z (b^{-1}z)^k b^{-1}y$.

In each subcase, let $p_i$ be the intersection point of $\beta \cap (x \cup y \cup z \cup b)$
corresponding to the last generator of $w_i$.
Let $s_i$ be the short arc of $(x \cup y \cup z \cup b) - \alpha$ containing $p_i$.
Let $R_i$ be the disk region of $S - (\alpha \cup s_i)$ which contains $P'$.

Let $G$ be a compressing disk in $W - K$ disjoint from $F$.
Let $\overline{w}$ be the word of $\partial G$ in $x,y,z$.
By lemma \ref{lem5}, $\overline{w}$ is freely reduced to the empty word.
In any case, since $\alpha$ contains subarcs connecting $x$ to $z$ and $y$ to $z$,
a cancellation of two adjacent generators of $\overline{w}$ is $zz^{-1}$ or $z^{-1}z$.
The cancellation and the subsequent cancellations take place in one of the following three ways.

\begin{itemize}
\item[(i)] in the interior of $R_i$.
\item[(ii)] when $\partial G $ passes through $s_i$.
\item[(iii)] in the complementary region of $R_i$.
\end{itemize}
However, the arc $\alpha$ is an obstruction to $\partial G$, and
we will show that actually $G$ cannot exist in any case.

In the interior of $R_i$, only a $k$-times nested cancellation of $zz^{-1}$ occurs in $\overline{w}$.
See Figure \ref{fig11} for an example of $k=2$.
The left of Figure \ref{fig11} is Case (a) and
the right is Case (b).
Subwords of $\overline{w}$ in the left and right are
$z(z(zz^{-1})z^{-1})x$ and $z(z(zz^{-1})z^{-1})x^{-1}$ respectively.
The thick solid blue curve represents the first cancellation $zz^{-1}$ and
the two thin solid blue curves represent the subsequent cancellation $z(zz^{-1})z^{-1}$.
The dotted curve represents that a cancellation does not occur any more.

\begin{figure}[htb]
\begin{center}
\includegraphics[width=14.5cm]{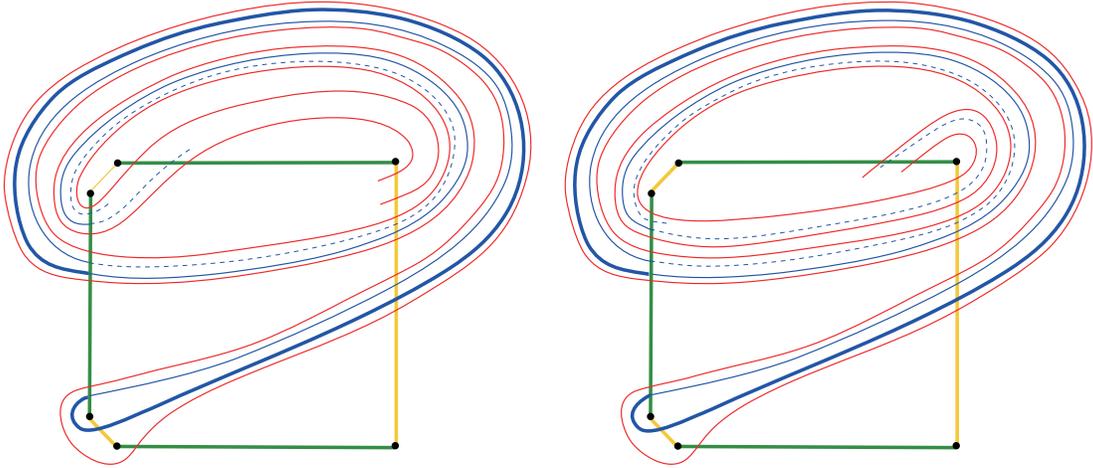}
\caption{A cancellation in the interior of $R_i$}\label{fig11}
\end{center}
\end{figure}

When $\partial G$ passes through $s_i$, a cancellation does not occur as can be seen in Figure \ref{fig12}.
The left of Figure \ref{fig12} is when $w_i$ ends with $y^{-1}b$ (Subcases $1$ and $3$) and
the right is when $w_i$ ends with $b^{-1}y$ (Subcases $2$ and $4$).
Subwords of $\overline{w}$ depicted in Figure \ref{fig12} are $zy^{-1}z^{-1}$ and $zyz^{-1}$,
hence a cancellation does not occur.

\begin{figure}[ht!]
\begin{center}
\includegraphics[width=14.5cm]{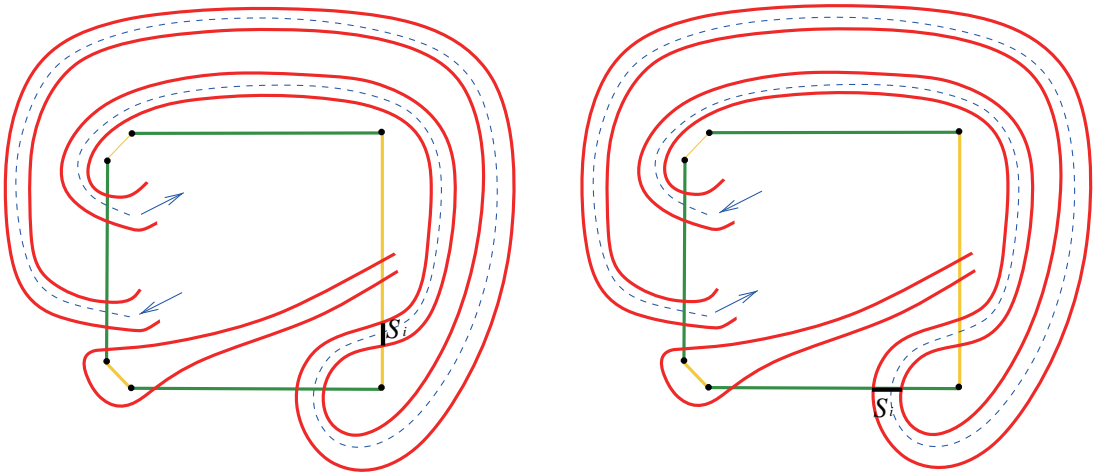}
\caption{No cancellation when $\partial G$ passes through $s_i$}\label{fig12}
\end{center}
\end{figure}

Now consider a cancellation in the complementary region of $R_i$.
See Figure \ref{fig13} -- \ref{fig16} for an example of $k=2$.
(For simplicity, we draw $\beta$ instead of $\alpha$.)
In Subcase $1$, the subword of $\overline{w}$ is $z^{-1}(z^{-1}(z^{-1}z)z)y^{-1}$.
So only a $2$-times nested cancellation of $z^{-1}z$ occurs.
The subwords of $\overline{w}$ in Subcases $2$, $3$, $4$ are
$z^{-1}(z^{-1}(z^{-1}z)z)y^{\pm 1}$,
$z^{-1}(z^{-1}(z^{-1}z)z)y^{\pm 1}$,
$z^{-1}(z^{-1}(z^{-1}z)z)y$, respectively.
In general, only a $k$-times nested cancellation of $z^{-1}z$ occurs in $\overline{w}$.
So we conclude that $G$ cannot exist.

\begin{figure}[htb]
\begin{center}
\includegraphics[width=8.7cm]{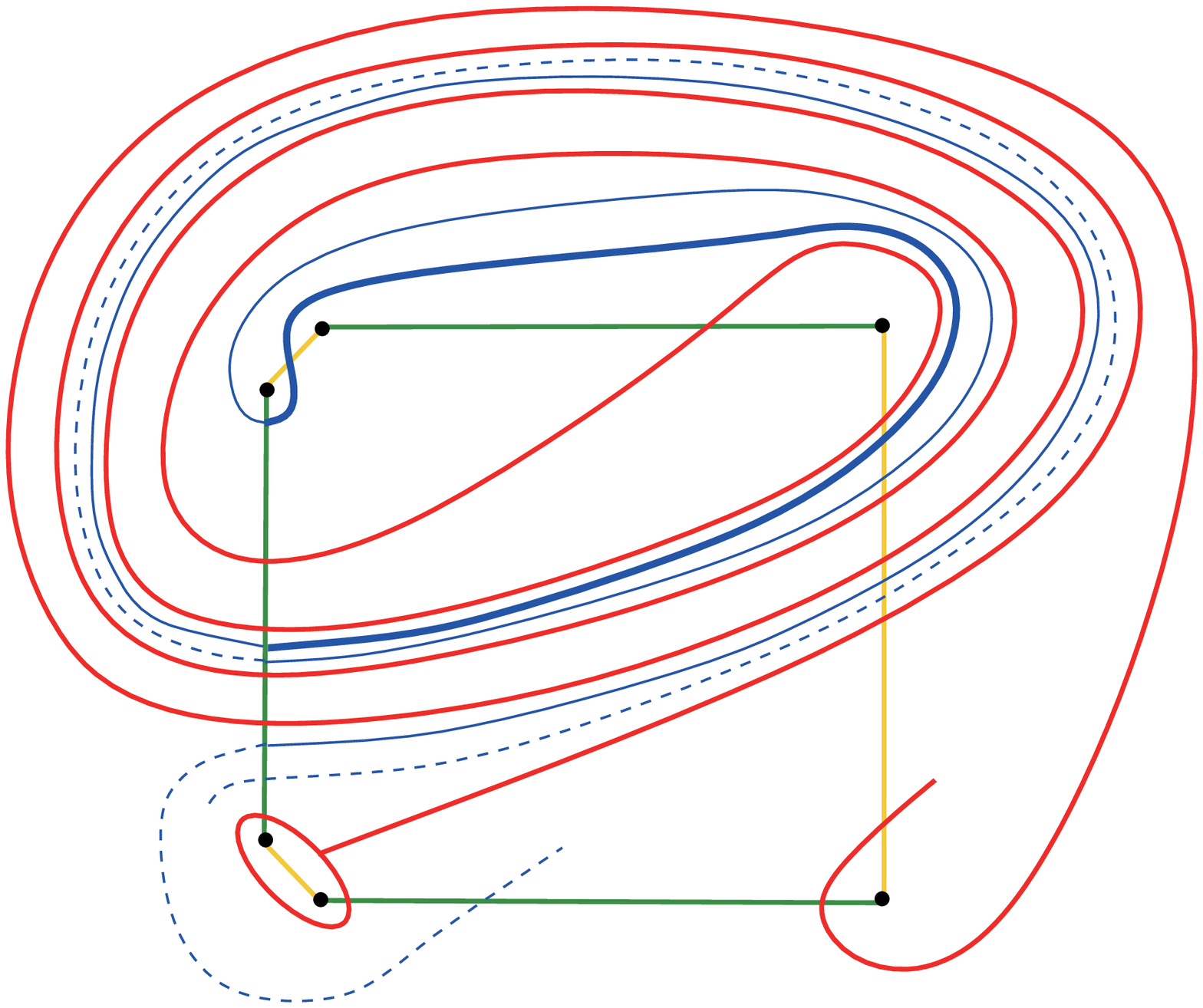}
\caption{A cancellation in the complementary region of $R_i$--- Subcase $1$}\label{fig13}
\end{center}

\end{figure}
\begin{figure}[ht!]
\begin{center}
\includegraphics[width=8.7cm]{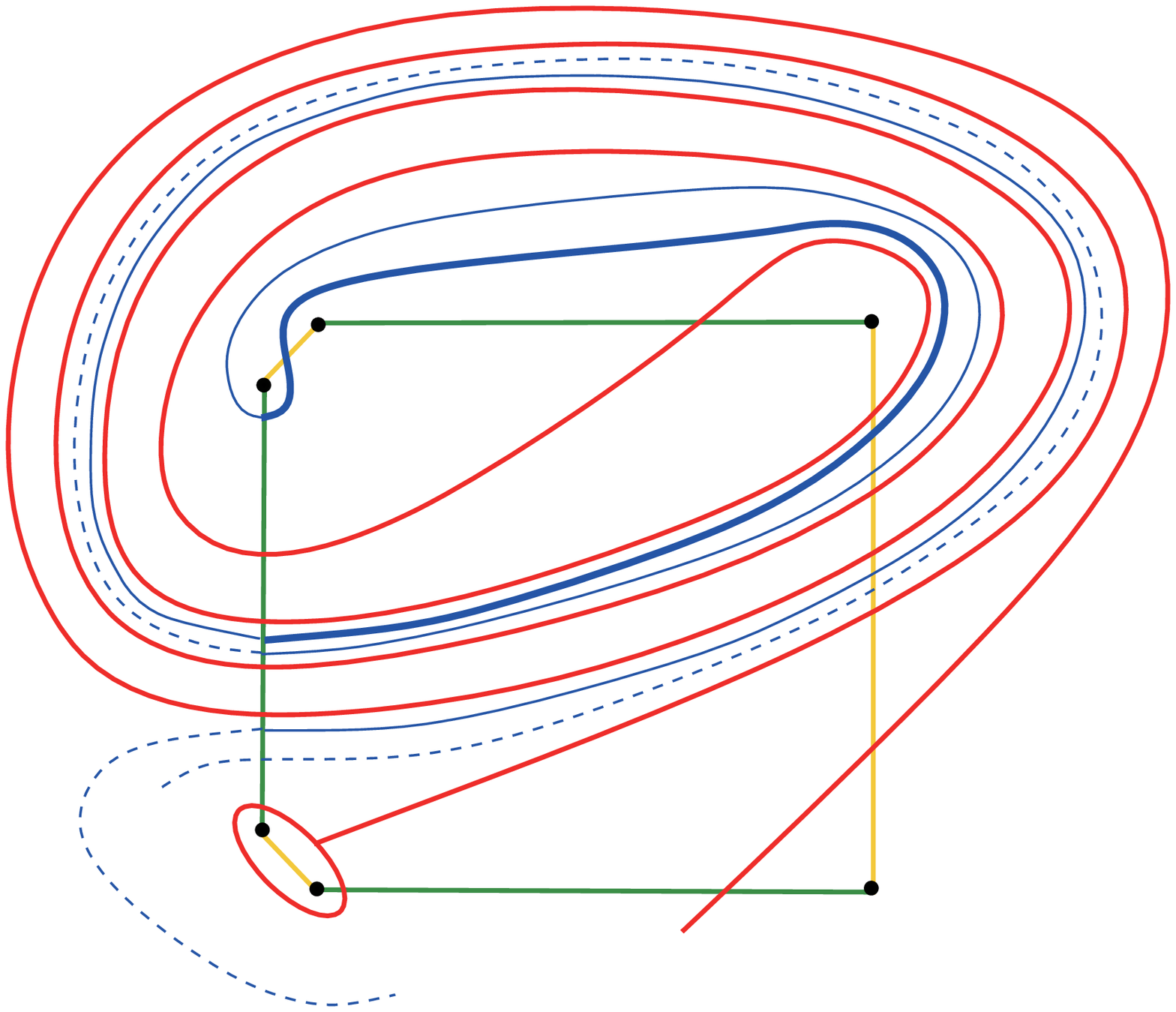}
\caption{A cancellation in the complementary region of $R_i$--- Subcase $2$}\label{fig14}
\end{center}
\end{figure}

\begin{figure}[ht!]
\begin{center}
\includegraphics[width=8.7cm]{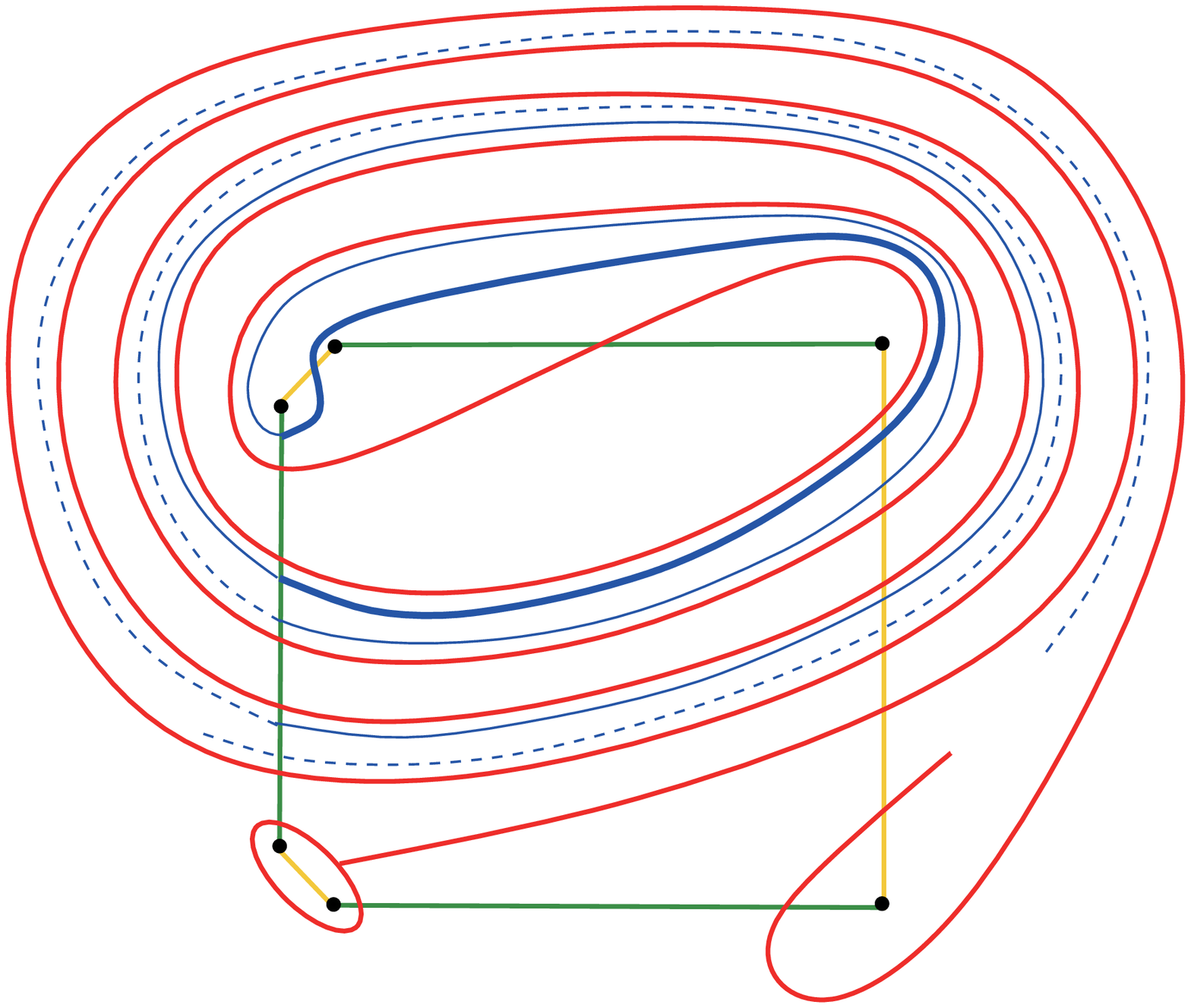}
\caption{A cancellation in the complementary region of $R_i$--- Subcase $3$}\label{fig15}
\end{center}
\end{figure}

\begin{figure}[ht!]
\begin{center}
\includegraphics[width=8.7cm]{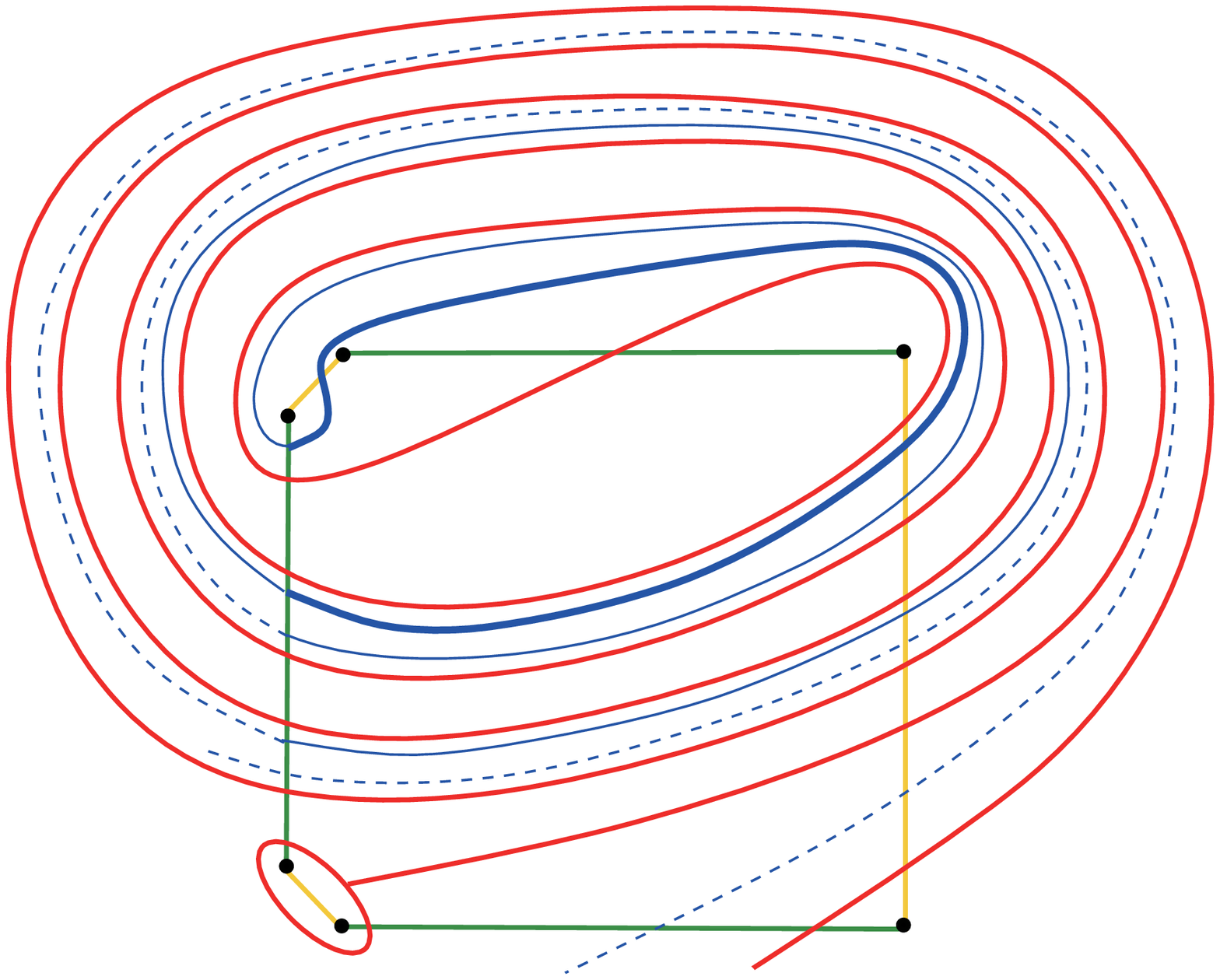}
\caption{A cancellation in the complementary region of $R_i$--- Subcase $4$}\label{fig16}
\end{center}
\end{figure}

\vspace{0.2cm}

Finally we show that there exists a weak reducing disk $D'$ in $V - K$ disjoint from $D$ and $C$.
Choose a weak reducing disk $D'$ disjoint from $D$ using Lemma \ref{lem2}.
If $C$ is disjoint from $D'$, then we are done.
Suppose that $C \cap D' \ne \emptyset$.
Consider an outermost arc $\gamma$ of $C \cap D'$ in $C$ such that
the outermost disk $C'$ in $C$ cut off by $\gamma$ is disjoint from $D$.
One of the disks obtained by surgery of $D'$ along $C'$, denoted by $D''$,
is isotopic to neither $D$ nor $D'$.
(The other is isotopic to $D$.)
By the argument of the preceding special case, $D''$ is a weak reducing disk.
Note that $|C \cap D''| < |C \cap D'|$.
By repeating the argument with $D''$ instead of $D'$,
we get the desired weak reducing disk disjoint from $D$ and $C$.

%\noindent {\bf Acknowledgements.}
%The second author is supported by the Basic Science Research Program
%through the National Research Foundation of Korea (NRF),
%funded by the Ministry of Education (2015R1D1A1A01056953).

\end{document}